\documentclass{article}

\usepackage{myarxiv}
\usepackage[utf8]{inputenc} 
\usepackage[T1]{fontenc}    
\usepackage{hyperref}       
\usepackage{url}            
\usepackage{booktabs}       
\usepackage{amsfonts}       
\usepackage{nicefrac}       
\usepackage{microtype}      
\usepackage{lipsum}

\usepackage{amssymb,amsmath,amsfonts,amsthm,enumerate}
\usepackage{subfigure}
\usepackage{graphicx}
\usepackage{epsfig}
\usepackage{float}
\usepackage[usenames]{color}
\usepackage{cite}
\usepackage{bm}

\newtheorem{corollary}{Corollary}
\newtheorem{example}{Example}
\newtheorem{lemma}{Lemma}
\newtheorem{problem}{Problem}
\newtheorem{remark}{Remark}
\newtheorem{theorem}{Theorem}

\allowdisplaybreaks[4]

\def\BC{\mathbb C}
\def\BN{\mathbb N}
\def\BR{\mathbb R}

\def\cF{\mathcal F}

\def\cL{\mathcal L}
\def\cP{\mathcal P}
\def\cU{\mathcal U}
\def\F{\mathrm F}
\def\I{\mathrm I}
\def\T{\mathrm T}

\def\rd{\mathrm d}

\def\e{\mathrm e}

\def\ri{\mathrm i}

\def\supp{\mathrm{supp}}

\def\Ga{\Gamma}

\def\Om{\Omega}
\def\al{\alpha}
\def\be{\beta}
\def\ga{\gamma}
\def\de{\delta}

\def\ve{\varepsilon}

\def\ze{\zeta}
\def\ka{\kappa}
\def\la{\lambda}

\def\vp{\varphi}

\def\f{\frac}
\def\nb{\nabla}
\def\ov{\overline}
\def\pa{\partial}

\def\wt{\widetilde}

\title{Inverse Moving Source Problem for Fractional Diffusion(-Wave) Equations: Determination of Orbits}

\author{
  Guanghui \uppercase{Hu}\\
  Beijing Computational Science Research Center,\\
  Building 9, East Zone, ZPark II, No.\! 10 Xibeiwang East Road, Haidian District, Beijing 100093, China.\\
  \texttt{hu@csrc.ac.cn}\\
  \And
  Yikan \uppercase{Liu}\thanks{Corresponding author.}\\
  Research Institute for Electronic Science, Hokkaido University,\\
  N12W7, Kita-Ward, Sapporo 060-0812, Japan.\\
  \texttt{ykliu@es.hokudai.ac.jp}\\
  \AND
  Masahiro \uppercase{Yamamoto}\\
  Graduate School of Mathematical Sciences, The University of Tokyo\\
  3-8-1 Komaba, Meguro-ku, Tokyo 153-8914, Japan;\\
  Honorary Member of Academy of Romanian Scientists\\
  Splaiul Independentei Street, No.\! 54, 050094 Bucharest, Romania;\\
  Peoples' Friendship University of Russia (RUDN University)\\
  6 Miklukho-Maklaya Street, Moscow 117198, Russian Federation.\\
  \texttt{myama@ms.u-tokyo.ac.jp} \\
}

\begin{document}
\maketitle

\begin{abstract}
This paper is concerned with the inverse problem on determining the orbit of an moving source in fractional diffusion(-wave) equations either in a connected bounded domain of $\BR^d$ or in the whole space $\BR^d$. Based on a newly established fractional Duhamel's principle, we derive a Lipschitz stability estimate in the case of a localized moving source by the observation data at $d$ interior points. The uniqueness for the general non-localized moving source is verified with additional data of more interior observations.
\end{abstract}

\MRsubject{74B05\and 35B60}

\section{Introduction}

Let $T>0$ and $\Om\subset\BR^d$, $d=1,2,3,\ldots$, be either a connected bounded domain with a smooth boundary $\pa\Om$ or $\Om=\BR^d$. For $0<\al\le2$, consider an initial(-boundary) value problem for a time-fractional diffusion(-wave) equation
\begin{equation}\label{eq-IBVP-u}
\begin{cases}
(\pa_t^\al+\cL)u(\bm x,t)=g(\bm x-\bm\ga(t)), & {\bm x}\in\Om,\ 0<t<T,\\
\begin{cases}
u(\bm x,0)=0 & \mbox{if }0<\al\le1,\\
u(\bm x,0)=\pa_t u(\bm x,0)=0 & \mbox{if }1<\al\le2,
\end{cases} &{\bm x}\in\Om,\\
u(\bm x,t)=0\quad\mbox{if }\Om\mbox{ is bounded}, & {\bm x}\in\pa\Om,\ 0<t<T.
\end{cases}
\end{equation}
Here, $\cL$ is an elliptic operator with respect to the spatial variable $\bm x\in\BR^d$, and $\pa_t^\al$ denotes the Caputo derivative with respect to the time variable $t\in\BR_+:=(0,\infty)$, which will be precisely defined in Section \ref{sec:2}. The function $g$ is an approximation of Dirac's delta function in $\BR^d$, and $\bm\ga:\BR_+\longrightarrow\BR^d$ describes the orbit of a moving source in $\BR^d$. The governing equation in \eqref{eq-IBVP-u} is called a (time-fractional) diffusion equation when $\al\in(0,1]$, whereas is called a (time-fractional) diffusion-wave equation or a (time-fractional) wave equation when $\al\in(1,2]$. Hence, the system \eqref{eq-IBVP-u} approximates a moving point source problem for the (time-fractional) diffusion(-wave) equation.

In this paper, we are interested in the inverse moving source problem of recovering an unknown orbit function $\bm\ga(t)$ from the solution data detected at a finite number of interior receivers. More precisely, we investigate the following problem.

\begin{problem}[Inverse moving source problem]\label{prob-IMSP}
Let $u$ be the solution to \eqref{eq-IBVP-u} and pick $N$ interior points $\bm x^j\in\Om\ (j=1,\ldots,N)$. Provided that the source profile $g(\bm x)$ is suitably given, determine the source orbit $\bm\ga(t)\ (0\le t\le T)$ by the multiple point observations of $u$ at $\{\bm x^j\}_{j=1}^N\times[0,T]$.
\end{problem}

We remark that the relation between the orbit function and the received dynamical signals is nonlinear, whereas the operator which maps the source profile function $g$ to the forward solution $u$ is linear. Hence, the problem considered in this paper is a nonlinear inverse issue.

We refer to Isakov \cite{Isakov89} for an overview of uniqueness and stability results on inverse source problems. The approaches of applying Carleman estimates and the unique continuation of evolutionary equations have been widely used in the literature and have led to uniqueness and stability results for both inverse coefficient and inverse source problems with the dynamical data over a finite time (see e.g., \cite{ACY-EJAM04,BK1981,CM2006,K1992,JLY2017,Ya1995,Ya1999} as an incomplete list). The concept of increasing stability was explored in \cite{Isakov07} and later investigated further for inverse source problems in \cite{CIL2016}. For stationary (non-moving) sources, the uniqueness in determining source positions with boundary surface data was deduced in \cite{EH}, and those for upper and lower estimates of source positions were derived in \cite{KY01,KY05} in one and  higher dimensions.

To the best of the authors' knowledge, literature on inverse moving source problems arising in fractional diffusion(-wave) equations is rather limited and even remain open. A logarithmic stability and an iterative inversion scheme were considered in \cite{LZ17} for recovering temporal source terms in fractional diffusion equations. Using the moment theory, a uniqueness result to inverse moving source problems in electromagnetism was proved (see \cite{HKLZ}) using boundary surface data. In a series of works \cite{OIO11,N12,O19}, numerical algorithms were examined for reconstructing a moving orbit from boundary data of solutions of the scalar wave equation.

The aim of this paper is to derive stability and uniqueness results with a finite number of interior monitoring points for the fractional model \eqref{eq-IBVP-u} with $\al\in(0,2]$. Our arguments rely on the fixed point theory similarly to \cite{SY11,FK16,WLL16} but are modified to be applicable to the system \eqref{eq-IBVP-u}. For this purpose, we have deduced Duhamel's principle for time-fractional partial differential equations of any order (see Lemma \ref{lem-Duhamel}) and a uniform solution representation for the fractional order $\alpha\in(0,2]$ via the Fourier transform in the whole space (see Lemma \ref{lem-FP-V}). These lemmas have generalized the corresponding well-known results for equations with integral orders in bounded domains, making new contributions to the theory of fractional equations. The interior observation data are here formulated from the theoretical viewpoint, and in a forthcoming paper, on the basis of the current work, we will discuss the inverse problem with more practical data.

The remaining part of this paper is organized as follows. Preliminary knowledge and our main results for fractional equations in bounded and unbounded domains will be stated in Section \ref{sec:2}. Three auxiliary lemmas will be proved in Section \ref{app-proof}. The proofs of our stability result for localized {moving sources} (Theorem \ref{thm-local}) and the uniqueness for non-localized ones (Corollary \ref{cor-global}) will be carried out in Section \ref{sec:4}.

\section{Preliminaries and Main Results}\label{sec:2}

Throughout this paper, we set $\BN:=\{1,2,3,\ldots\}$, and by $C>0$ we denote generic constants which may change from line to line. For $\be\in\BR$, denote the largest integer smaller than or equal to $\be$ by the floor function $\lfloor\be\rfloor$, and the smallest integer larger than or equal to $\be$ by the ceiling function $\lceil\be\rceil$. For $\be\in[0,1]$, define the Riemann-Liouville integral operator $J^\be$ by
\[
J^\be f(t):=\left\{\!
\begin{alignedat}{2}
& f(t), & \quad & \be=0,\\
& \f1{\Ga(\be)}\int_0^t\f{f(s)}{(t-s)^{1-\be}}\,\rd s, & \quad & 0<\be\le1,
\end{alignedat}
\right.\quad f\in C[0,\infty),
\]
where $\Ga(\,\cdot\,)$ is the Gamma function. Then for $\be\in\BR_+$, the Caputo derivative $\pa_t^\be$ and the Riemann-Liouville derivative $D_t^\be$ are defined as
\[
\pa_t^\be:=J^{\lceil\be\rceil-\be}\circ\f{\rd^{\lceil\be\rceil}}{\rd t^{\lceil\be\rceil}},\quad D_t^\be:=\f{\rd^{\lceil\be\rceil}}{\rd t^{\lceil\be\rceil}}\circ J^{\lceil\be\rceil-\be},
\]
where $\circ$ denotes the composition. For the solution representation, we recall the familiar Mittag-Leffler function
\[
E_{\be,\mu}(z):=\sum_{\ell=0}^\infty\f{z^\ell}{\Ga(\be\ell+\mu)},\quad\be\in\BR_+,\ \mu\in\BR,\ z\in\BC,
\]
which satisfies the frequently used estimate (see Podlubny \cite[Theorem 1.5]{P99})
\begin{equation}\label{eq-est-ML}
|E_{\be,\mu}(z)|\le\f C{1+|z|},\quad\be\in(0,2),\ \mu\in\BR_+,\ \f{\pi\be}2<|\mathrm{arg}\,z|\le\pi.
\end{equation}
For later use, we state the following formula concerning the Riemann-Liouville derivative and Mittag-Leffler functions.

\begin{lemma}\label{lem-RL-ML}
For $\be\in\BR_+$ and $\la\in\BR$, we have
\[
D_t^{\lceil\be\rceil-\be}\left(t^{\lceil\be\rceil-1}E_{\be,\lceil\be\rceil}(\la\,t^\be)\right)=t^{\be-1}E_{\be,\be}(\la\,t^\be).
\]
\end{lemma}

Next, we generalize useful Duhamel's principle to time-fractional evolution equations with arbitrary orders $\be\in\BR_+$.

\begin{lemma}[Fractional Duhamel's principle]\label{lem-Duhamel}
Let $\be\in\BR_+$ and $\Om\subset\BR^d$ be a domain. Let $F:\Om\times(0,T)\longrightarrow\BR$ be a smooth function, and $\cP$ be a linear partial differential operator with respect to $\bm x$ defined in $\Om$ whose coefficients are independent of $t$. If a smooth function $u$ satisfies
\begin{equation}\label{eq-Duhamel-u}
\begin{cases}
(\pa_t^\be+\cP)u=F & \mbox{in }\Om\times(0,T),\\
\pa_t^m u=0\ (m=0,\ldots,\lceil\be\rceil-1) & \mbox{in }\Om\times\{0\},
\end{cases}
\end{equation}
then $u$ allows the representation
\begin{equation}\label{eq-Duhamel}
u(\,\cdot\,,t)=\int_0^t D_t^{\lceil\be\rceil-\be}v(\,\cdot\,,t;s)\,\rd s,\quad0<t<T,
\end{equation}
where $v(\,\cdot\,,\,\cdot\,;s)$ is a smooth function satisfying the following homogeneous equation with a parameter $s\in(0,T)$:
\begin{equation}\label{eq-Duhamel-v}
\begin{cases}
(\pa_t^\be+\cP)v=0 & \mbox{in }\Om\times(s,T),\\
\pa_t^m v=0\ (m=0,\ldots,\lceil\be\rceil-2),\ \pa_t^{\lceil\be\rceil-1}v=F(\,\cdot\,,s) & \mbox{in }\Om\times\{s\}.
\end{cases}
\end{equation}
Here we automatically interpret
\begin{align*}
D_t^\be v(\,\cdot\,,t;s) & =\f{\pa_t^{\lceil\be\rceil}}{\Ga(\lceil\be\rceil-\be)}\int_s^t\f{v(\,\cdot\,,r;s)}{(t-r)^{\be-\lfloor\be\rfloor}}\,\rd r,\\
\pa_t^\be v(\,\cdot\,,t;s) & =\f1{\Ga(\lceil\be\rceil-\be)}\int_s^t\f{\pa_r^{\lceil\be\rceil}v(\,\cdot\,,r;s)}{(t-r)^{\be-\lfloor\be\rfloor}}\,\rd r
\end{align*}
for $\be\not\in\BN$, since $v(\,\cdot\,,t;s)$ is only defined for $t<s$.
\end{lemma}

The above lemma generalizes similar results in \cite{LRY16,L17}, where the source term was assumed to take the form of separated variables. For other fractional varieties of Duhamel's principle, we refer to Umarov and Saidamatov \cite{US07}, Zhang and Xu \cite{ZX11}.

In the sequel, all vectors are by default column vectors unless specified otherwise. For instance, we write $\bm x=(x_1,\ldots,x_d)^\T\in\BR^d$ and $\nb f=(\pa_1f,\ldots,\pa_d f)^\T$, where $(\,\cdot\,)^\T$ stands for the transpose and $\pa_k=\f\pa{\pa x_k}$ ($k=1,\ldots,d$). The inner product in $\BR^d$ is denoted by $\bm x\cdot\bm y$, and the Euclidean distance $|\cdot|$ is induced as $|\bm x|=(\bm x\cdot\bm x)^{1/2}$. Given a matrix $\bm\Psi=(\psi_{jk})\in\BR^{d\times d}$, the $\ell^2$ norm $|\,\cdot\,|$ and the Frobenius norm $\|\cdot\|_\F$ of $\bm\Psi$ are defined as
\[
|\bm\Psi|:=\max_{|\bm x|=1}|\bm\Psi\bm x|,\quad\|\bm\Psi\|_\F:=\left(\sum_{j,k=1}^d\psi_{jk}^2\right)^{1/2}.
\]
By the norm equivalence in finite dimensional vector spaces, there exists a constant $C>0$ such that
\begin{equation}\label{eq-norm}
C^{-1}|\bm\Psi|\le\|\bm\Psi\|_\F\le C|\bm\Psi|,\quad\forall\,\bm\Psi\in\BR^{d\times d}.
\end{equation}
The open ball centered at $\bm x\in\BR^d$ with radius $r>0$ is denoted by $B_r(\bm x):=\{\bm y\in\BR^d\mid|\bm y-\bm x|<r\}$, and especially we abbreviate $B_r(\bm0)=B_r$.

For a domain $\Om\subset\BR^d$, denote the usual $L^2$ inner product by $(\,\cdot\,,\,\cdot\,)$, and let $H^p(\Om)$ ($p\in\BR$) be the $L^2$-based Sobolev spaces (see Adams \cite{A75}). If $\Om$ is a connected bounded domain, then the elliptic operator $\cL$ in the initial-boundary value problem \eqref{eq-IBVP-u} is defined as
\[
\cL:H^2(\Om)\cap H_0^1(\Om)\longrightarrow L^2(\Om),\quad f\longmapsto-\mathrm{div}(\bm A\nb f)+c\,f,
\]
where the matrix $\bm A=(a_{jk})_{1\le j,k\le d}\in(C^1(\ov\Om))^{d\times d}$ is symmetric and strictly positive definite uniformly on $\ov\Om$, and $c\in L^\infty(\Om)$ is non-negative. In this case, the operator $\cL$ generates an eigensystem $\{(\la_n,\vp_n)\}_{n\in\BN}$ such that $\cL\vp_n=\la_n\vp_n$ in $\Om$. Moreover, it is well known that $0<\la_1<\la_2\le\cdots$, $\la_n\to\infty$ as $n\to\infty$, and $\{\vp_n\}\subset H^2(\Om)\cap H_0^1(\Om)$ forms a complete orthonormal system of $L^2(\Om)$.

For a square integrable function $f\in L^2(\BR^d)$, denote its Fourier transform by
\[
\wt f(\bm\xi)=(\cF f)(\bm\xi):=\f1{(2\pi)^{d/2}}\int_{\BR^d}f(\bm x)\,\e^{-\ri\,\bm x\cdot\bm\xi}\,\rd\bm x.
\]
For $p\in\BR$, the $H^p(\BR^d)$ norm can be represented by using Fourier transform as
\[
\|f\|_{H^p(\BR^d)}=\left(\int_{\BR^d}(1+|\bm\xi|^2)^p|\wt f(\bm\xi)|^2\,\rd\bm\xi\right)^{1/2},\quad f\in H^p(\BR^d).
\]
If $\Om=\BR^d$, then we define the elliptic operator $\cL$ in the initial value problem \eqref{eq-IBVP-u} as
\[
\cL:H^2(\BR^d)\longrightarrow L^2(\BR^d),\quad f\longmapsto-\mathrm{div}(\bm A\nb f)+\bm b\cdot\nb f+c\,f,
\]
where we assume that the matrix $\bm A=(a_{jk})_{1\le j,k\le d}$, the vector $\bm b=(b_1,\ldots,b_d)^\T$ and the scalar $c$ are constants, and $\bm A$ is strictly positive definite. Especially, only in the case of $\al=2$ we additionally assume $\bm b=\bm0$ and $c\ge0$.

Regarding the initial(-boundary) value problem
\begin{equation}\label{eq-IBVP-V}
\begin{cases}
(\pa_t^\al+\cL)V=0 & \mbox{in }\Om\times\BR_+,\\
\begin{cases}
V=v_0 & \mbox{if }0<\al\le1,\\
V=0,\ \pa_t V=v_1 & \mbox{if }1<\al\le2
\end{cases} & \mbox{in }\Om\times\{0\},\\
V=0\quad\mbox{if }\Om\mbox{ is bounded} & \mbox{on }\pa\Om\times\BR_+,
\end{cases}
\end{equation}
we provide the well-posedness results in the following lemma.

\begin{lemma}\label{lem-FP-V}
Let $\Om\subset\BR^d$ be either a connected bounded domain or $\Om=\BR^d$, and $v_0\in H^p(\Om),v_1\in H^{p-2/\al}(\Om)$ for some fixed $p\in\BR$. Then there exists a unique solution $V\in C([0,T];H^p(\Om))$ to \eqref{eq-IBVP-V}. Moreover, if $\Om$ is bounded, then the solution to \eqref{eq-IBVP-V} takes the form
\begin{equation}\label{eq-sol-IBVP-V}
V(\bm x,t)=\sum_{n=1}^\infty(v_{\lceil\al\rceil-1},\vp_n)\,t^{\lceil\al\rceil-1}E_{\al,\lceil\al\rceil}(-\la_n t^\al)\vp_n(\bm x).
\end{equation}
If $\Om=\BR^d$, then the solution to \eqref{eq-IBVP-V} takes the form
\begin{equation}\label{eq-sol-IVP-V}
V(\bm x,t)=\left\{\!\begin{alignedat}{2}
& \f{t^{\lceil\al\rceil-1}}{(2\pi)^{d/2}}\int_{\BR^d}\wt v_{\lceil\al\rceil-1}(\bm\xi)E_{\al,\lceil\al\rceil}(-S(\bm\xi)t^\al)\,\e^{\ri\,\bm\xi\cdot\bm x}\,\rd\bm\xi, & \quad & 0<\al<2,\\
& \f1{(2\pi)^{d/2}}\int_{\BR^d}\wt v_1(\bm\xi)S(\bm\xi)^{-1/2}\sin(S(\bm\xi)^{1/2}t)\,\e^{\ri\,\bm\xi\cdot\bm x}\,\rd\bm\xi, & \quad & \al=2,
\end{alignedat}\right.
\end{equation}
where $S(\bm\xi):=\bm A\bm\xi\cdot\bm\xi+\ri\,\bm b\cdot\bm\xi+c$.
\end{lemma}

In a bounded domain, the solution representation by the eigensystem is well {known} (see Sakamoto and Yamamoto \cite{SY11}). However, solutions to \eqref{eq-IBVP-V} in the whole space $\BR^d$ seem not well investigated to the best of our knowledge, and we refer to Eidelman and Kochubei \cite{EK04} for the fundamental solution. In such a sense, formula \eqref{eq-sol-IVP-V} in Lemma \ref{lem-FP-V} gives a novel solution representation via the Fourier transform.

Now we are well prepared to discuss Problem \ref{prob-IMSP}. We begin with specifying the choices of the source profile $g(\bm x)$ and the orbit $\bm\ga(t)$. Assume that $g$ is smooth and compactly supported, i.e., $g\in C_0^\infty(\Om)$ and there exists a constant $\de>0$ such that $\supp\,g\subset B_\de$. A typical choice of $g$ can be the following bell-shaped function
\begin{equation}\label{eq-bell}
g(\bm x)=\left\{\!\begin{alignedat}{2}
& C\exp\left(\f1{|\bm x|^2-\de^2}\right), & \quad & |\bm x|<\de,\\
& 0, & \quad & |\bm x|\ge\de.
\end{alignedat}\right.
\end{equation}
For the unknown $\bm\ga$, basically we restrict it in the admissible set
\begin{equation}\label{eq-def-U0}
\cU_0:=\{\bm\ga\in(C^\infty[0,T])^d\mid\bm\ga(0)=\bm0,\ \|\bm\ga'\|_{C[0,T]}\le K,\ \bm\ga(t)+\supp\,g\subset\Om,\ 0\le t\le T\},
\end{equation}
where $K>0$ is a constant. In other words, we restrict our consideration in such orbits that they are smooth and start from the origin with a maximum velocity.

First we investigate a special case of a localized moving source. More precisely, for a sufficiently small $\ve>0$, we further restrict the unknown orbit in
\begin{equation}\label{eq-def-U1}
\cU_1:=\{\bm\ga\in\cU_0\mid\|\bm\ga\|_{C[0,T]}\le\ve\},
\end{equation}
which means $\{\bm\ga(t)\}_{0\le t\le T}\subset B_\ve$ for all $\bm\ga\in\cU_1$.

Since there are $d$ components in the orbit, it is natural to take at least $d$ observation points for the {unique} identification. Within the admissible set $\cU_1$, we pick the minimum necessary $d$ observation points $\bm x^j$ ($j=1,\ldots,d$) and make the following key assumption: there exists a constant $C>0$ depending on $g$, $\{\bm x^j\}_{j=1}^d$ and $\ve$ such that
\begin{equation}\label{eq-assume-obs}
\left|\begin{pmatrix}
\nb g(\bm y^1) & \nb g(\bm y^2) & \cdots & \nb g(\bm y^d)
\end{pmatrix}^{-1}\right|\le C,\quad\forall\,\bm y^j\in\ov{B_\ve(\bm x^j)}\,,\ j=1,\ldots,d.
\end{equation}
In other words, we assume that the matrix $(\nb g(\bm y^1)\ \cdots\ \nb g(\bm y^d))$ is invertible for all $(\bm y^1,\ldots,\bm y^d)\in\prod_{j=1}^d\ov{B_\ve(\bm x^j)}\,$.

\begin{example}
We rephrase assumptions \eqref{eq-def-U1} and \eqref{eq-assume-obs} in the case of $d=1$. For any $\ga\in\cU_1$, we have $\ga(0)=0$ and $|\ga(t)|\le\ve$, $|\ga'(t)|\le K$ for $0\le t\le T$. As for the observation point $x^1$, the assumption \eqref{eq-assume-obs} means
\[
|g'(y)|\ge C^{-1}>0,\quad\forall\,y\in[x^1-\ve,x^1+\ve],
\]
which implies $[x^1-\ve,x^1+\ve]\subset\supp\,g'$.
\end{example}

Now we can state Lipschitz stability and uniqueness results for Problem \ref{prob-IMSP} with the observation data taken at $\{\bm x^j\}_{j=1}^d\times[0,T]$.

\begin{theorem}\label{thm-local}
Fix $\bm\ga_1,\bm\ga_2\in\cU_1$, where $\cU_1$ was defined by \eqref{eq-def-U1}. Denote by $u_1$ and $u_2$ the solutions to \eqref{eq-IBVP-u} with $\bm\ga=\bm\ga_1$ and $\bm\ga=\bm\ga_2$, respectively. If the set of observation points $\{\bm x^j\}_{j=1}^d$ satisfies \eqref{eq-assume-obs}, then there exists a constant $C>0$ depending on $g,$ $\{\bm x^j\}_{j=1}^d$ and $\cU_1$ such that
\[
\|\bm\ga_1-\bm\ga_2\|_{C[0,T]}\le C\sum_{j=1}^d\|\pa_t^\al(u_1-u_2)(\bm x^j,\,\cdot\,)\|_{C[0,T]}.
\]
Especially, $u_1(\bm x^j,\,\cdot\,)=u_2(\bm x^j,\,\cdot\,)\ (j=1,\ldots,d)$ on $[0,T]$ implies $\bm\ga_1=\bm\ga_2$ on $[0,T]$.
\end{theorem}

Our main result Theorem \ref{thm-local} requires condition \eqref{eq-assume-obs} for $\bm x^j$, $j=1,2,\ldots,N$, and especially the number $N$ of the monitoring points $\bm x^j$ should be at least $d$ which is the spatial dimensions. This is reasonable because as unknowns we have to determine $d$ components of $\bm\ga(t)$, and our data are $N$ functions in $t\in[0,T]$.

The key to proving the above theorem is reducing the original problem to a vector-valued Volterra integral equation of the second kind with respect to the difference $\bm\ga_1-\bm\ga_2$. To this end, the representations of solutions to \eqref{eq-IBVP-u} are essential, where Lemmas \ref{lem-Duhamel} and \ref{lem-FP-V} play important roles. Such an argument is also witnessed in \cite{SY11,FK16,WLL16} which also rely on similar non-vanishing assumptions as \eqref{eq-assume-obs}. Nevertheless, due to the nonlinearity of our problem with respect to the orbits, assumption \eqref{eq-assume-obs} looks more complicated than that in \cite{SY11,FK16,WLL16}.

Remarkably, the constant $C$ in the stability estimate of Theorem \ref{thm-local} does not depend on the order $\al\in(0,2]$. Indeed, such a uniform estimate of Lipschitz type is achieved at the cost of accessing the $\al$th order derivative of the observation data. Meanwhile, one can also see from the proof that the ill-posedness resulted from $\al$ is overwhelmed by the key assumption \eqref{eq-assume-obs} along with the admissible set $\cU_1$.

In Theorem \ref{thm-local}, the Lipschitz stability with minimum possible observation points is achieved within the admissible set $\cU_1$ in \eqref{eq-def-U1}, which is rather restrictive. Moreover, since \eqref{eq-assume-obs} implies $\bm x^j\in\supp(\nb g)$ ($j=1,\ldots,d$), the required observation condition seems also strict in practice. On the opposite direction, we can remove the localization assumption $\|\bm\ga\|_{C[0,T]}\le\ve$ in \eqref{eq-def-U1} and obtain a uniqueness result at the cost of very dense observation points.

\begin{corollary}\label{cor-global}
Fix $\bm\ga_1,\bm\ga_2\in\cU_0$, where $\cU_0$ was defined by \eqref{eq-def-U0}. Denote by $u_1$ and $u_2$ the solutions to \eqref{eq-IBVP-u} with $\bm\ga=\bm\ga_1$ and $\bm\ga=\bm\ga_2$, respectively. Assume that there exist a finite set of observation points $X:=\{\bm x^j\}_{j=1}^N$ and a constant $\ve>0$ such that for any $\bm y\in\Om\cap\ov{B_{K T}}\,$, there exist $d$ observation points $\{\bm x^j(\bm y)\}_{j=1}^d\subset X\cap B_\de(\bm y)$ and a constant $C>0$ such that
\begin{equation}\label{eq-assume-obs'}
\left|\begin{pmatrix}
\nb g(\bm z^1) & \cdots & \nb g(\bm z^d)
\end{pmatrix}^{-1}\right|\le C,\quad\forall\,\bm z^j\in\ov{B_\ve(\bm x^j(\bm y)-\bm y)}\,,\ j=1,\ldots,d.
\end{equation}
Then the relation $u_1(\bm x^j,\,\cdot\,)=u_2(\bm x^j,\,\cdot\,)\ (j=1,\ldots,N)$ on $[0,T]$ implies $\bm\ga_1=\bm\ga_2$ on $[0,T]$.
\end{corollary}

As one can imagine, the above corollary follows from the repeated application of Theorem \ref{thm-local}, where the invertibility assumption \eqref{eq-assume-obs'} is a generalization of \eqref{eq-assume-obs}. It suffices to restrict $\bm y\in\Om$ in the ball $\ov{B_{K T}}$ because $\|\bm\ga\|_{C[0,T]}\le K T$ for any $\bm\ga\in\cU_0$ by the definition \eqref{eq-def-U0} of $\cU_0$. Since $\Om\cap\ov{B_{K T}}$ is bounded, the number $N$ of observation points can definitely be finite.

\begin{example}
In the one-dimensional case, if $g$ takes the form of a bell-shaped function \eqref{eq-bell}, then it is readily seen that a choice of $\ve$ and $X$ in Corollary $\ref{cor-global}$ can be
\[
\ve=\f\de9,\quad\bm x^j=\f{(-1)^j\lfloor j/2\rfloor\de}4,\quad N=\lceil4(K T+\de)/\de\rceil.
\]
\end{example}

\section{Proofs of Lemmas \ref{lem-RL-ML}--\ref{lem-FP-V}}\label{app-proof}

\begin{proof}[Proof of Lemma $\ref{lem-RL-ML}$]
By the definitions of Mittag-Leffler functions and the Riemann-Liouville derivative, we direct calculate
\begin{align*}
D_t^{\lceil\be\rceil-\be}\left(t^{\lceil\be\rceil-1}E_{\be,\lceil\be\rceil}(\la\,t^\be)\right) & =\f1{\Ga(\be-\lfloor\be\rfloor)}\sum_{\ell=0}^\infty\f{\la^\ell}{\Ga(\be\ell+\lceil\be\rceil)}\f\rd{\rd t}\int_0^t\f{s^{\be\ell+\lceil\be\rceil-1}}{(t-s)^{\lceil\be\rceil-\be}}\,\rd s\\
& =\f1{\Ga(\be-\lfloor\be\rfloor)}\sum_{\ell=0}^\infty\f{\Ga(\be\ell+\lceil\be\rceil)\Ga(\be-\lfloor\be\rfloor)\la^\ell(t^{\be(\ell+1)})'}{\Ga(\be\ell+\lceil\be\rceil)\Ga(\be\ell+\lceil\be\rceil+\be-\lfloor\be\rfloor)}\\
& =\sum_{\ell=0}^\infty\f{\be(\ell+1)\la^\ell t^{\be(\ell+1)-1}}{\Ga(\be(\ell+1)+1)}=t^{\be-1}\sum_{\ell=0}^\infty\f{(\la\,t^\be)^\ell}{\Ga(\be\ell+\be)}=t^{\be-1}E_{\be,\be}(\la\,t^\be),
\end{align*}
where we have used the formula $\Ga(\be+1)=\be\,\Ga(\be)$.
\end{proof}

\begin{proof}[Proof of Lemma $\ref{lem-Duhamel}$]
The case of $\be\in\BN$ is straightforward and we only give a proof for the case of $\be\not\in\BN$. Actually, it suffices to verify that the function $u$ defined by \eqref{eq-Duhamel} satisfies \eqref{eq-Duhamel-u}. Since $F,u,v$ are assumed to be smooth, we can take any derivatives when needed.

First, it follows from the definition of the Riemann-Liouville derivative that
\begin{equation}\label{eq-Duhamel-pf1}
D_t^{\lceil\be\rceil-\be}F(\,\cdot\,,t)=\f1{\Ga(\be-\lfloor\be\rfloor)}\f\pa{\pa t}\int_0^t\f{F(\,\cdot\,,t-s)}{s^{\lceil\be\rceil-\be}}\,\rd s=\f1{\Ga(\be-\lfloor\be\rfloor)}\left(\f{F(\,\cdot\,,0)}{t^{\lceil\be\rceil-\be}}+\int_0^t\f{\pa_s F(\,\cdot\,,s)}{(t-s)^{\lceil\be\rceil-\be}}\,\rd s\right).
\end{equation}
Next, from \eqref{eq-Duhamel-u} we calculate
\begin{align*}
\pa_t u(\,\cdot\,,t) & =D_t^{\lceil\be\rceil-\be}v(\,\cdot\,,t;t)+\int_0^t\pa_t D_t^{\lceil\be\rceil-\be}v(\,\cdot\,,t;s)\,\rd s\\
& =\left\{\!\begin{alignedat}{2}
& D_t^{1-\be}F(\,\cdot\,,t)+\int_0^t D_t^{1-\be}\pa_t v(\,\cdot\,,t;s)\,\rd s, & \quad & 0<\be<1,\\
& \int_0^t D_t^{\lceil\be\rceil-\be}\pa_t v(\,\cdot\,,t;s)\,\rd s, & \quad & \be>1,
\end{alignedat}\right.
\end{align*}
where we used the initial condition at $t=s$ in \eqref{eq-Duhamel-v}. Inductively, we obtain
\begin{equation}\label{eq-diff-u}
\pa_t^m u(\,\cdot\,,t)=\left\{\!\begin{alignedat}{2}
& \int_0^t D_t^{\lceil\be\rceil-\be}\pa_t^m v(\,\cdot\,,t;s)\,\rd s, & \quad & m<\lceil\be\rceil,\\
& D_t^{\lceil\be\rceil-\be}F(\,\cdot\,,t)+\int_0^t D_t^{\lceil\be\rceil-\be}\pa_t^{\lceil\be\rceil}v(\,\cdot\,,t;s)\,\rd s, & \quad & m=\lceil\be\rceil.
\end{alignedat}\right.
\end{equation}
Since $v$ is sufficiently smooth, for $m<\lceil\be\rceil$ we pass $t\to0$ in \eqref{eq-diff-u} to find
\[
\pa_t^m u(\,\cdot\,,0)=0,\quad m=0,\ldots,\lceil\be\rceil-1,
\]
i.e., the function $u$ satisfies the initial condition in \eqref{eq-Duhamel-u}. Meanwhile, substituting \eqref{eq-diff-u} with $m=\lceil\be\rceil$ into the definition of the Caputo derivative gives
\begin{equation}\label{eq-Duhamel-pf2}
\pa_t^\be u(\,\cdot\,,t)=\f1{\Ga(\lceil\be\rceil-\be)}\int_0^t\f{\pa_s^{\lceil\be\rceil}u(\,\cdot\,,s)}{(t-s)^{\be-\lfloor\be\rfloor}}\,\rd s=:\f{\I_1+\I_2}{\Ga(\lceil\be\rceil-\be)},
\end{equation}
where
\[
\I_1:=\int_0^t\f{D_s^{\lceil\be\rceil-\be}F(\,\cdot\,,s)}{(t-s)^{\be-\lfloor\be\rfloor}}\,\rd s,\quad\I_2:=\int_0^t\f1{(t-s)^{\be-\lfloor\be\rfloor}}\int_0^s D_s^{\lceil\be\rceil-\be}\pa_s^{\lceil\be\rceil}v(\,\cdot\,,s;r)\,\rd r\rd s.
\]
For $\I_1$, the application of \eqref{eq-Duhamel-pf1} yields
\begin{align}
\I_1 & =\f1{\Ga(\be-\lfloor\be\rfloor)}\int_0^t\f1{(t-s)^{\be-\lfloor\be\rfloor}}\left(\f{F(\,\cdot\,,0)}{s^{\lceil\be\rceil-\be}}+\int_0^s\f{\pa_r F(\,\cdot\,,r)}{(s-r)^{\lceil\be\rceil-\be}}\,\rd r\right)\rd s\nonumber\\
& =\f1{\Ga(\be-\lfloor\be\rfloor)}\left(F(\,\cdot\,,0)\int_0^t\f{\rd s}{(t-s)^{\be-\lfloor\be\rfloor}s^{\lceil\be\rceil-\be}}+\int_0^t\pa_r F(\,\cdot\,,r)\int_r^t\f{\rd s}{(t-s)^{\be-\lfloor\be\rfloor}(s-r)^{\lceil\be\rceil-\be}}\,\rd r\right)\nonumber\\
& =\Ga(\lceil\be\rceil-\be)\left(F(\,\cdot\,,0)+\int_0^t\pa_r F(\,\cdot\,,r)\,\rd r\right)=\Ga(\lceil\be\rceil-\be)\,F(\,\cdot\,,t).
\end{align}
For $\I_2$, by suitably exchanging the order of integration, we utilize the definition of Caputo and Riemann-Liouville derivatives to calculate
\begin{align}
\I_2 & =\int_0^t\f1{(t-s)^{\be-\lfloor\be\rfloor}}\int_0^s D_s^{\lceil\be\rceil-\be}\pa_s^{\lceil\be\rceil}v(\,\cdot\,,s;r)\,\rd r\rd s\nonumber\\
& =\f1{\Ga(\be-\lfloor\be\rfloor)}\int_0^t\f1{(t-s)^{\be-\lceil\be\rceil}}\int_0^s\pa_s\int_0^{s-r}\f{\pa_s^{\lceil\be\rceil}v(\,\cdot\,,s-\tau;r)}{\tau^{\lceil\be\rceil-\be}}\,\rd\tau\rd r\rd s\nonumber\\
& =\f1{\Ga(\be-\lfloor\be\rfloor)}\int_0^t\f1{(t-s)^{\be-\lceil\be\rceil}}\int_0^s\f{\pa_r^{\lceil\be\rceil}v(\,\cdot\,,r;r)}{(s-r)^{\lceil\be\rceil-\be}}\,\rd r\rd s\nonumber\\
& \quad\,+\f1{\Ga(\be-\lfloor\be\rfloor)}\int_0^t\f1{(t-s)^{\be-\lceil\be\rceil}}\int_0^s\!\!\int_r^s\f{\pa_\tau^{\lceil\be\rceil+1}v(\,\cdot\,,\tau;r)}{(s-\tau)^{\lceil\be\rceil-\be}}\,\rd\tau\rd r\rd s\nonumber\\
& =\f1{\Ga(\be-\lfloor\be\rfloor)}\int_0^t\pa_r^{\lceil\be\rceil}v(\,\cdot\,,r;r)\int_r^t\f{\rd s}{(t-s)^{\be-\lfloor\be\rfloor}(s-r)^{\lceil\be\rceil-\be}}\,\rd r\nonumber\\
& \quad\,+\f1{\Ga(\be-\lfloor\be\rfloor)}\int_0^t\!\!\int_0^\tau\pa_\tau^{\lceil\be\rceil+1}v(\,\cdot\,,\tau;r)\int_\tau^t\f{\rd s}{(t-s)^{\be-\lfloor\be\rfloor}(s-\tau)^{\lceil\be\rceil-\be}}\,\rd r\rd\tau\nonumber\\
& =\Ga(\lceil\be\rceil-\be)\left(\int_0^t\pa_r^{\lceil\be\rceil}v(\,\cdot\,,r;r)\,\rd r+\int_0^t\!\!\int_0^\tau\pa_\tau^{\lceil\be\rceil+1}v(\,\cdot\,,\tau;r)\,\rd r\rd\tau\right).
\end{align}
On the other hand, since the operator $\cP$ is independent of $t$, we calculate $-\cP u$ as
\begin{align}
-\cP u(\,\cdot\,,t) & =-\int_0^t D_t^{\lceil\be\rceil-\be}\cP v(\,\cdot\,,t;s)\,\rd s=\int_0^t D_t^{\lceil\be\rceil-\be}\pa_t^\be v(\,\cdot\,,t;s)\,\rd s\nonumber\\
& =\int_0^t\f{\pa_t}{\Ga(\be-\lfloor\be\rfloor)}\int_s^t\f1{(t-r)^{\lceil\be\rceil-\be}}\f1{\Ga(\lceil\be\rceil-\be)}\int_s^r\f{\pa_\tau^{\lceil\be\rceil}v(\,\cdot\,,\tau;s)}{(r-\tau)^{\be-\lfloor\be\rfloor}}\,\rd\tau\rd r\rd s\nonumber\\
& =\f1{\Ga(\be-\lfloor\be\rfloor)\Ga(\lceil\be\rceil-\be)}\int_0^t\pa_t\int_0^{t-s}\f1{r^{\lceil\be\rceil-\be}}\int_s^{t-r}\f{\pa_\tau^{\lceil\be\rceil}v(\,\cdot\,,\tau;s)}{(t-r-\tau)^{\be-\lfloor\be\rfloor}}\,\rd\tau\rd r\rd s\nonumber\\
& =\f1{\Ga(\be-\lfloor\be\rfloor)\Ga(\lceil\be\rceil-\be)}\int_0^t\!\!\int_0^{t-s}\f{\pa_t}{r^{\lceil\be\rceil-\be}}\int_s^{t-r}\f{\pa_\tau^{\lceil\be\rceil}v(\,\cdot\,,\tau;s)}{(t-r-\tau)^{\be-\lfloor\be\rfloor}}\,\rd\tau\rd r\rd s\nonumber\\
& =\f1{\Ga(\be-\lfloor\be\rfloor)\Ga(\lceil\be\rceil-\be)}\left\{\int_0^t\!\!\int_0^{t-s}\f{\pa_s^{\lceil\be\rceil}v(\,\cdot\,,s;s)}{r^{\lceil\be\rceil-\be}(t-s-r)^{\be-\lfloor\be\rfloor}}\,\rd r\rd s+\I_3\right\}\nonumber\\
& =\int_0^t\pa_s^{\lceil\be\rceil}v(\,\cdot\,,s;s)\,\rd s+\f{\I_3}{\Ga(\be-\lfloor\be\rfloor)\Ga(\lceil\be\rceil-\be)},
\end{align}
where
\begin{align}
\I_3 & :=\int_0^t\!\!\int_0^{t-s}\f1{r^{\lceil\be\rceil-\be}}\int_0^{t-s-r}\f{\pa_t^{\lceil\be\rceil+1}v(\,\cdot\,,t-r-\tau;s)}{\tau^{\be-\lfloor\be\rfloor}}\,\rd\tau\rd r\rd s\nonumber\\
&=\int_0^t\!\!\int_s^t\f1{(r-s)^{\lceil\be\rceil-\be}}\int_r^t\f{\pa_t^{\lceil\be\rceil+1}v(\,\cdot\,,t-\tau+s;s)}{(\tau-r)^{\be-\lfloor\be\rfloor}}\,\rd\tau\rd r\rd s\nonumber\\
& =\int_0^t\!\!\int_0^\tau\pa_t^{\lceil\be\rceil+1}v(\,\cdot\,,t-\tau+s;s)\int_s^\tau\f{\rd r}{(\tau-r)^{\be-\lfloor\be\rfloor}(r-s)^{\lceil\be\rceil-\be}}\,\rd s\rd\tau\nonumber\\
& =\Ga(\be-\lfloor\be\rfloor)\Ga(\lceil\be\rceil-\be)\int_0^t\!\!\int_0^\tau\pa_t^{\lceil\be\rceil+1}v(\,\cdot\,,t-\tau+s;s)\,\rd s\rd\tau\nonumber\\
& =\Ga(\be-\lfloor\be\rfloor)\Ga(\lceil\be\rceil-\be)\int_0^t\!\!\int_0^{t-s}\pa_\tau^{\lceil\be\rceil+1}v(\,\cdot\,,\tau+s;s)\,\rd\tau\rd s\nonumber\\
& =\Ga(\be-\lfloor\be\rfloor)\Ga(\lceil\be\rceil-\be)\int_0^t\!\!\int_0^\tau\pa_\tau^{\lceil\be\rceil+1}v(\,\cdot\,,\tau;s)\,\rd s\rd\tau.\label{eq-Duhamel-pf3}
\end{align}
The combination of \eqref{eq-Duhamel-pf2}--\eqref{eq-Duhamel-pf3} immediately indicates
\begin{align*}
(\pa_t^\be+\cP)u(\,\cdot\,,t) & =\f{\I_1+\I_2}{\Ga(\lceil\be\rceil-\be)}-\int_0^t\pa_s^{\lceil\be\rceil}v(\,\cdot\,,s;s)\,\rd s-\f{\I_3}{\Ga(\be-\lfloor\be\rfloor)\Ga(\lceil\be\rceil-\be)}\\
& =F(\,\cdot\,,t)+\int_0^t\!\!\int_0^\tau\pa_\tau^{\lceil\be\rceil+1}v(\,\cdot\,,\tau;r)\,\rd r\rd\tau-\f{\I_3}{\Ga(\be-\lfloor\be\rfloor)\Ga(\lceil\be\rceil-\be)}=F(\,\cdot\,,t).
\end{align*}
Therefore, it is verified that the function $u$ defined by \eqref{eq-Duhamel} indeed satisfies \eqref{eq-Duhamel-u}, and the proof of Lemma \ref{lem-Duhamel} is completed.
\end{proof}

\begin{proof}[Proof of Lemma $\ref{lem-FP-V}$]
If $\Om$ is a bounded domain, the results follow immediately by the same argument as that in Sakamoto and Yamamoto \cite{SY11}. Henceforth we only deal with the unbounded case of $\Om=\BR^d$.

Recalling the definition of $S(\bm\xi)$ in Lemma \ref{lem-FP-V}, formally we have $\cF(\cL V(\,\cdot\,,t))(\bm\xi)=S(\bm\xi)\wt V(\bm\xi,t)$. Then taking Fourier transform in \eqref{eq-IBVP-V} with respect to the spatial variables yields a fractional ordinary differential equation with a parameter $\bm\xi$:
\[
\begin{cases}
(\pa_t^\al+S(\bm\xi))\wt V(\bm\xi,t)=0, & t>0,\\
\begin{cases}
\wt V(\bm\xi,0)=\wt v_0(\bm\xi) & \mbox{if }0<\al\le1,\\
\wt V(\bm\xi,0)=0,\ \pa_t\wt V(\bm\xi,0)=\wt v_1(\bm\xi) & \mbox{if }1<\al\le2.
\end{cases}
\end{cases}
\]
The solution to the above equation turns out to be
\begin{equation}\label{eq-sol-ODE}
\wt V(\bm\xi,t)=\begin{cases}
\wt v_{\lceil\al\rceil-1}(\bm\xi)t^{\lceil\al\rceil-1}E_{\al,\lceil\al\rceil}(-S(\bm\xi)t^\al), & 0<\al<2,\\
\wt v_1(\bm\xi)S(\bm\xi)^{-1/2}\sin(S(\bm\xi)^{1/2}t), & \al=2,
\end{cases}
\end{equation}
where $S(\bm\xi)\ge0$ for $\al=2$ because we assumed $\bm b=\bm0$ and $c\ge0$ in this case.

For any fixed $t\ge0$, our aim is to verify the boundedness of $\|V(\,\cdot\,,t)\|_{H^p(\BR^d)}$. In the case of $0<\al<2$, we have to estimate $|E_{\al,\lceil\al\rceil}(-S(\bm\xi)t^\al)|$. Denoting by $\ka>0$ the smallest eigenvalue of the strict positive definite matrix $\bm A$, we see
\[
\mathrm{Re}\,S(\bm\xi)=\bm A\bm\xi\cdot\bm\xi+c\ge\f\ka2|\bm\xi|^2\mbox{ for }|\bm\xi|\gg1,\quad|\mathrm{Im}\,S(\bm\xi)|=|\bm b\cdot\bm\xi|\le|\bm b||\bm\xi|.
\]
Hence, there exists a constant $R=R(\al)>0$ such that $\f{\pi\al}2<|\mathrm{arg}(-S(\bm\xi)t^\al)|\le\pi$ for $|\bm\xi|\ge R$. Then we can employ \eqref{eq-est-ML} to estimate
\[
|E_{\al,\lceil\al\rceil}(-S(\bm\xi)t^\al)|\le\f C{1+|S(\bm\xi)|t^\al}\le\f C{1+|\bm\xi|^2t^\al},\quad|\bm\xi|\ge R.
\]
For $|\bm\xi|<R$, it is readily seen that $|E_{\al,\lceil\al\rceil}(-S(\bm\xi)t^\al)|$ is uniformly bounded.

For $0<\al\le1$, we divide $\BR^d=B_R\cup(\BR^d\setminus B_R)$ and use \eqref{eq-sol-ODE} to estimate
\begin{align*}
\|V(\,\cdot\,,t)\|_{H^p(\BR^d)}^2 & =\left(\int_{B_R}+\int_{\BR^d\setminus B_R}\right)(1+|\bm\xi|^2)^p|\wt v_0(\bm\xi)|^2|E_{\al,1}(-S(\bm\xi)t^\al)|^2\,\rd\bm\xi\\
& \le C\int_{B_R}(1+|\bm\xi|^2)^p|\wt v_0(\bm\xi)|^2\,\rd\bm\xi+\int_{\BR^d\setminus B_R}(1+|\bm\xi|^2)^p|\wt v_0(\bm\xi)|^2\left(\f C{1+|\bm\xi|^2t^\al}\right)^2\,\rd\bm\xi\\
& \le C\int_{\BR^d}(1+|\bm\xi|^2)^p|\wt v_0(\bm\xi)|^2\,\rd\bm\xi=C\|v_0\|_{H^p(\BR^d)}^2.
\end{align*}
For $1<\al<2$, we utilize the same argument as above and the uniform boundedness of $\f{\ze^{1/\al}}{1+\ze}$ for $\ze\ge0$ to deduce
\begin{align*}
\|V(\,\cdot\,,t)\|_{H^p(\BR^d)}^2 & =t^2\left(\int_{B_R}+\int_{\BR^d\setminus B_R}\right)(1+|\bm\xi|^2)^p|\wt v_1(\bm\xi)|^2|E_{\al,2}(-S(\bm\xi)t^\al)|^2\,\rd\bm\xi\\
& \le C\,t^2\int_{B_R}(1+|\bm\xi|^2)^{2/\al}(1+|\bm\xi|^2)^{p-2/\al}|\wt v_1(\bm\xi)|^2\,\rd\bm\xi\\
& \quad\,+\int_{\BR^d\setminus B_R}(1+|\bm\xi|^2)^{p-2/\al}|\wt v_1(\bm\xi)|^2\left(\f{C((1+|\bm\xi|^2)t^\al)^{1/\al}}{1+|\bm\xi|^2t^\al}\right)^2\,\rd\bm\xi\\
& \le C(t^2+1)\int_{\BR^d}(1+|\bm\xi|^2)^{p-2/\al}|\wt v_1(\bm\xi)|^2\,\rd\bm\xi=C(t^2+1)\|v_1\|_{H^{p-2/\al}(\BR^d)}^2.
\end{align*}

In the case of $\al=2$, thanks to the assumption $\bm b=\bm0$ and $c\ge0$, we have $S(\bm\xi)\ge\ka|\bm\xi|^2$. Then we estimate \eqref{eq-sol-ODE} as
\begin{align*}
\|V(\,\cdot\,,t)\|_{H^p(\BR^d)}^2 & =\left(\int_{B_1}+\int_{\BR^d\setminus B_1}\right)\f{(1+|\bm\xi|^2)^p|\wt v_1(\bm\xi)||\sin(S(\bm\xi)^{1/2}t)|}{S(\bm\xi)}\,\rd\bm\xi\\
& \le\int_{B_1}(1+|\bm\xi|^2)^{p-1}|\wt v_1(\bm\xi)|^2(1+|\bm\xi|^2)\f{S(\bm\xi)t^2}{S(\bm\xi)}\,\rd\bm\xi+\int_{\BR^d\setminus B_1}(1+|\bm\xi|^2)^{p-1}|\wt v_1(\bm\xi)|^2\f{1+|\bm\xi|^2}{\ka|\bm\xi|^2}\,\rd\bm\xi\\
& \le C\,t^2\int_{B_1}(1+|\bm\xi|^2)^{p-1}|\wt v_1(\bm\xi)|^2\,\rd\bm\xi+C\int_{\BR^d\setminus B_1}(1+|\bm\xi|^2)^{p-1}|\wt v_1(\bm\xi)|^2\,\rd\bm\xi\\
& \le C(t^2+1)\|v_1\|_{H^{p-1}(\BR^d)}^2.
\end{align*}
The proof of Lemma \ref{lem-FP-V} is completed.
\end{proof}

\section{Proofs of the Main Results}\label{sec:4}

\begin{proof}[Proof of Theorem $\ref{thm-local}$]
Let $u_1,u_2$ be the solutions to \eqref{eq-IBVP-u} with orbits $\bm\ga_1,\bm\ga_2\in\cU_1$, respectively. Setting $w:=u_1-u_2$, it is easy to observe that $w$ satisfies the following initial(-boundary) value problem
\begin{equation}\label{eq-IVP-w}
\begin{cases}
(\pa_t^\al+\cL)w=G & \mbox{in }\Om\times(0,T),\\
\begin{cases}
w=0 & \mbox{if }0<\al\le1,\\
w=\pa_t w=0 & \mbox{if }1<\al\le2
\end{cases} & \mbox{in }\Om\times\{0\},\\
w=0\quad\mbox{if }\Om\mbox{ is bounded} & \mbox{on }\pa\Om\times(0,T),
\end{cases}
\end{equation}
where $G(\bm x,t):=g(\bm x-\bm\ga_1(t))-g(\bm x-\bm\ga_2(t))$. According to the mean value theorem, there exists a smooth function $\bm\eta:\Om\times(0,T)\longrightarrow\BR^d$ such that
\[
G(\bm x,t)=\nb g(\bm\eta(\bm x,t))\cdot\bm\rho(t)=\sum_{k=1}^d G_k(\bm x,t)\rho_k(t),
\]
where $\bm\eta(\bm x,t)$ is a point lying on the segment between $\bm x-\bm\ga_1(t)$ and $\bm x-\bm\ga_2(t)$, and
\[
\bm\rho:=\bm\ga_2-\bm\ga_1=(\rho_1,\ldots,\rho_d)^\T,\quad G_k(\bm x,t):=\pa_k g(\bm\eta(\bm x,t))\ (k=1,\ldots,d).
\]
Substituting the observation points $\bm x=\bm x^j$ ($j=1,\ldots,d$) into the governing equation of \eqref{eq-IVP-w}, we obtain
\begin{equation}\label{eq-linear1}
\sum_{k=1}^d G_k(\bm x^j,t)\rho_k(t)=\pa_t^\al w(\bm x^j,t)+\cL w(\bm x^j,t),\quad j=1,\ldots,d.
\end{equation}

In order to give a representation of $\cL w(\bm x^j,t)$, we take advantage of Lemma \ref{lem-Duhamel} to write $\cL w$ as
\begin{equation}\label{eq-rep-w}
\cL w(\,\cdot\,,t)=\int_0^t D_t^{\lceil\al\rceil-\al}\cL v(\,\cdot\,,t;s)\,\rd s,\quad0<t\le T,
\end{equation}
where $v$ satisfies the following homogeneous initial(-boundary) value problem with a parameter $s\in(0,T)$:
\[
\begin{cases}
(\pa_t^\al+\cL)v=0 & \mbox{in }\Om\times(s,T),\\
\begin{cases}
v=G(\,\cdot\,,s) & \mbox{if }0<\al\le1,\\
v=0,\ \pa_t v=G(\,\cdot\,,s) & \mbox{if }1<\al\le2
\end{cases} & \mbox{in }\Om\times\{s\},\\
v=0\quad\mbox{if }\Om\mbox{ is bounded} & \mbox{on }\pa\Om\times(s,T).
\end{cases}
\]
In the case of a bounded domain $\Om$, it follows from \eqref{eq-sol-IBVP-V} that
\[
\cL v(\,\cdot\,,t;s)=\sum_{n=1}^\infty\la_n(G(\,\cdot\,,s),\vp_n)(t-s)^{\lceil\al\rceil-1}E_{\al,\lceil\al\rceil}(-\la_n(t-s)^\al)\vp_n.
\]
Using Lemma \ref{lem-RL-ML}, we substitute the above equality into \eqref{eq-rep-w} with $\bm x=\bm x^j$ to represent
\begin{align}
\cL w(\bm x^j,t) & =\int_0^t\sum_{n=1}^\infty\la_n(G(\,\cdot\,,s),\vp_n)(t-s)^{\al-1}E_{\al,\al}(-\la_n(t-s)^\al)\vp_n(\bm x^j)\,\rd s\nonumber\\
& =\int_0^t\sum_{k=1}^d q_{jk}(t,s)\rho_k(s)\,\rd s,\quad j=1,\ldots,d,\label{eq-rep-Aw}
\end{align}
where
\begin{equation}\label{eq-IBVP-q}
q_{jk}(t,s):=(t-s)^{\al-1}\sum_{n=1}^\infty\la_n(G_k(\,\cdot\,,s),\vp_n)\vp_n(\bm x^j)E_{\al,\al}(-\la_n(t-s)^\al).
\end{equation}
In the case of $\Om=\BR^d$, we turn to the Fourier transform to see $\cF(\cL v(\,\cdot\,,t;s))=S\,\wt v(\,\cdot\,,t;s)$, where we recall $S(\bm\xi)=\bm A\bm\xi\cdot\bm\xi+\ri\,\bm b\cdot\bm\xi+c$. Then it follows from \eqref{eq-sol-IVP-V} that
\[
\cF(\cL v(\,\cdot\,,t;s))(\bm\xi)=\begin{cases}
S(\bm\xi)\wt G(\bm\xi,s)(t-s)^{\lceil\al\rceil-1}E_{\al,\lceil\al\rceil}(-S(\bm\xi)t^\al), & 0<\al<2,\\
S(\bm\xi)^{1/2}\wt G(\bm\xi,s)\sin(S(\bm\xi)^{1/2}(t-s)), & \al=2.
\end{cases}
\]
Taking the inverse Fourier transform in the above equality and applying Lemma \ref{lem-RL-ML} to \eqref{eq-rep-w} again, we obtain
\begin{align*}
\cL w(\,\cdot\,,t) & =\cF^{-1}\left(\int_0^t D_t^{\lceil\al\rceil-\al}\cF(\cL v(\,\cdot\,,t;s))\,\rd s\right)\\
& =\left\{\!\begin{alignedat}{2}
& \cF^{-1}\left(\int_0^t S(\bm\xi)\wt G(\bm\xi,s)(t-s)^{\al-1}E_{\al,\al}(-S(\bm\xi)(t-s)^\al)\,\rd s\right), & \quad & 0<\al<2,\\
& \cF^{-1}\left(\int_0^t S(\bm\xi)^{1/2}\wt G(\bm\xi,s)\sin(S(\bm\xi)^{1/2}(t-s))\,\rd s\right), & \quad & \al=2.
\end{alignedat}\right.
\end{align*}
By $\wt G(\bm\xi,s)=\sum_{k=1}^d\wt G_k(\bm\xi,s)\rho_k(s)$ and taking $\bm x=\bm x^j$ ($j=1,\ldots,d$), again we arrive at the expression \eqref{eq-rep-Aw}, where $q_{jk}(t,s)$ is defined by
\begin{equation}\label{eq-IVP-q}
q_{jk}(t,s):=\begin{cases}
(t-s)^{\al-1}\cF^{-1}(S(\bm\xi)\wt G_k(\bm\xi,s)E_{\al,\al}(-S(\bm\xi)(t-s)^\al))(\bm x^j), & 0<\al<2,\\
\cF^{-1}(S(\bm\xi)^{1/2}\wt G_k(\bm\xi,s)\sin(S(\bm\xi)^{1/2}(t-s)))(\bm x^j), & \al=2.
\end{cases}
\end{equation}

Since the expression \eqref{eq-rep-Aw} is valid for both bounded and unbounded $\Om$, we plug \eqref{eq-rep-Aw} into \eqref{eq-linear1} and rewrite it in form of a linear system as
\begin{equation}\label{eq-linear2}
\bm P(t)\bm\rho(t)=\pa_t^\al\bm h(t)+\int_0^t\bm Q(t,s)\bm\rho(s)\,\rd s,
\end{equation}
where $\bm h(t):=(w(\bm x^1,t),\ldots,w(\bm x^d,t))^\T$ and
\[
\bm P(t):=(G_k(\bm x^j,t))_{1\le j,k\le d},\quad\bm Q(t,s):=(q_{jk}(t,s))_{1\le j,k\le d}
\]
are $d\times d$ matrices. Recalling the admissible set $\cU_1$ for $\bm\ga_1$ and $\bm\ga_2$, we see that $\bm\eta(\bm x^j,t)\in B_\ve(\bm x^j)$ for all $t\in[0,T]$. Therefore, by $G_k(\bm x,t)=\pa_k g(\bm\eta(\bm x,t))$, the key assumption \eqref{eq-assume-obs} indicates that the matrix
\[
\bm P(t)=\begin{pmatrix}
\nb g(\eta(\bm x^1,t)) & \nb g(\eta(\bm x^2,t)) & \cdots & \nb g(\eta(\bm x^d,t))
\end{pmatrix}^\T
\]
is invertible for all $t\in[0,T]$. In other words, there exists a constant $C>0$ such that
\begin{equation}\label{eq-inv-P}
|\bm P(t)^{-1}|\le C,\quad\forall\,t\in[0,T].
\end{equation}

As for the matrix $\bm Q(t,s)$, it suffices {to estimate $|q_{jk}(t,s)|$ appearing} in \eqref{eq-IBVP-q} and \eqref{eq-IVP-q} separately. In the case of \eqref{eq-IBVP-q}, the uniform boundedness of $E_{\al,\al}(-\ze)$ for $\ze\ge0$ yields
\begin{align*}
|q_{jk}(t,s)| & \le C(t-s)^{\al-1}\left|\sum_{n=1}^\infty\la_n(G_k(\,\cdot\,,s),\vp_n)\vp_n(\bm x^j)\right|\\
& =C(t-s)^{\al-1}|\cL G_k(\bm x^j,s)|=C(t-s)^{\al-1}|\cL\pa_k g(\bm\eta(\bm x^j,t))|\\
& \le C(t-s)^{\al-1}\|\cL\pa_k(g\circ\bm\eta)(\,\cdot\,,s)\|_{C(\ov\Om)}.
\end{align*}
Since $g$ is a given smooth function and $\bm\eta$ is also smooth and depends only on the admissible set $\cU_1$, it turns out that $\|\cL\pa_k(g\circ\bm\eta)\|_{C(\ov\Om\times[0,T])}$ are uniformly bounded for $k=1,\ldots,d$, implying
\begin{equation}\label{eq-est-q}
|q_{jk}(t,s)|\le C(t-s)^{\al-1},\quad1\le j,k\le d.
\end{equation}
For \eqref{eq-IVP-q}, we deal with the cases of $0<\al<2$ and $\al=2$ separately. For $0<\al<2$, the similar argument to that in the proof of Lemma \ref{lem-FP-V} guarantees a constant $R=R(\al)>0$ such that $\f{\pi\al}2<|\mathrm{arg}(-S(\bm\xi)t^\al)|\le\pi$ for $|\bm\xi|\ge R$. Then we employ \eqref{eq-est-ML} to estimate
\begin{align*}
|E_{\al,\al}(-S(\bm\xi)(t-s)^\al)| & \le\left\{\!\begin{alignedat}{2}
& C, & \quad & |\bm\xi|\le R,\\
& \f C{1+|S(\bm\xi)|(t-s)^\al}, & \quad & |\bm\xi|\ge R
\end{alignedat}\right.\\
& \le C,\quad\bm\xi\in\BR^d,\ 0\le s<t\le T.
\end{align*}
On the other hand, it is readily seen that $|S(\bm\xi)|\le|\bm A\bm\xi\cdot\bm\xi|+|\bm b\cdot\bm\xi|+|c|\le C(1+|\bm\xi|^2)$. Thus, based on the definition of the inverse Fourier transform, we can estimate
\begin{align}
|q_{jk}(t,s)| & =(t-s)^{\al-1}\left|\cF^{-1}\left(S(\bm\xi)\wt G_k(\bm\xi,s)E_{\al,\al}(-S(\bm\xi)(t-s)^\al)\right)(\bm x^j)\right|\nonumber\\
& \le\f{(t-s)^{\al-1}}{(2\pi)^{d/2}}\int_{\BR^d}|S(\bm\xi)||\wt G_k(\bm\xi,s)||E_{\al,\al}(-S(\bm\xi)(t-s)^\al)||\e^{\ri\,\bm\xi\cdot\bm x^j}|\,\rd\bm\xi\nonumber\\
& \le C(t-s)^{\al-1}\int_{\BR^d}\left((1+|\bm\xi|^2)^{(d+5)/4}|\wt G_k(\bm\xi,s)|\right)(1+|\bm\xi|^2)^{-(d+1)/4}\,\rd\bm\xi\nonumber\\
& \le C(t-s)^{\al-1}\left(\int_{\BR^d}(1+|\bm\xi|^2)^{(d+5)/2}|\wt G_k(\bm\xi,s)|^2\right)^{1/2}\left(\int_{\BR^d}\f{\rd\bm\xi}{(1+|\bm\xi|^2)^{(d+1)/2}}\right)^{1/2}\label{eq-ineq-CS}\\
& \le C\|G_k(\,\cdot\,,s)\|_{H^{(d+5)/2}(\BR^d)}(t-s)^{\al-1},\nonumber
\end{align}
where we used the Cauchy-Schwarz inequality in \eqref{eq-ineq-CS}. Since $g,\bm\eta$ are smooth and $g$ is compactly supported, we see that $G_k=\pa_k(g\circ\bm\eta)$ is also smooth and compactly supported, indicating the uniform boundedness of $\|G_k(\,\cdot\,,s)\|_{H^{(d+5)/2}(\BR^d)}$ for $0<s<T$ and $k=1,\ldots,d$. Therefore, again we arrive at \eqref{eq-est-q} in the case $\Om=\BR^d$ with $0<\al<2$. Finally, for $\al=2$ we estimate in the same manner as
\begin{align*}
|q_{jk}(t,s)| & \le\f1{(2\pi)^{d/2}}\int_{\BR^d}S(\bm\xi)^{1/2}|\wt G_k(\bm\xi,s)||\sin(S(\bm\xi)^{1/2}(t-s))||\e^{\ri\,\bm\xi\cdot\bm x^j}|\,\rd\bm\xi\\
& \le C(t-s)\int_{\BR^d}S(\bm\xi)|\wt G_k(\bm\xi,s)|\,\rd\bm\xi\le C(t-s)\int_{\BR^d}(1+|\bm\xi|^2)|\wt G_k(\bm\xi,s)|\,\rd\bm\xi\\
& \le C(t-s)\left(\int_{\BR^d}(1+|\bm\xi|^2)^{(d+5)/2}|\wt G_k(\bm\xi,s)|^2\,\rd\bm\xi\right)^{1/2}\\
& \le C(t-s)\|G_k(\,\cdot\,,s)\|_{H^{(d+5)/2}(\BR^d)}\le C(t-s),\quad1\le j,k\le d,
\end{align*}
which is consistent with \eqref{eq-est-q}. Consequently, it reveals that the upper bound \eqref{eq-est-q} holds for both cases of domains and remains valid for any $0<\al\le2$. This together with \eqref{eq-norm} implies the estimate
\[
|\bm Q(t,s)|\le C\|\bm Q(t,s)\|_\F=C\left(\sum_{j,k=1}^d|q_{jk}(t,s)|\right)^{1/2}\le C(t-s)^{\al-1}.
\]
The combination of \eqref{eq-linear2}, \eqref{eq-inv-P} and the above estimate yields
\[
|\bm\rho(t)|\le|\bm P(t)^{-1}|\left(|\pa_t^\al\bm h(t)|+\int_0^t|\bm Q(t,s)||\bm\rho(s)|\,\rd s\right)\le C\left(|\pa_t^\al\bm h(t)|+\int_0^t(t-s)^{\al-1}|\bm\rho(s)|\,\rd s\right).
\]
Eventually, we employ {Gr\"onwall's} inequality with a weakly singular kernel (see Henry \cite[Lemma 7.1.1]{H81}) to conclude
\begin{align*}
|\bm\rho(t)| & \le C\left(|\pa_t^\al\bm h(t)|+C\int_0^t\left.\f\rd{\rd\ze}E_{\al,1}(\ze^\al)\right|_{\ze=C(t-s)}|\pa_s^\al\bm h(s)|\,\rd s\right)\\
& \le C\|\pa_t^\al\bm h\|_{C[0,T]}\left(1+\int_0^t s^{\al-1}E_{\al,\al}(C s^\al)\,\rd s\right)\le C\|\pa_t^\al\bm h\|_{C[0,T]}
\end{align*}
for $0<t\le T$, and hence
\begin{align*}
\|\bm\ga_1-\bm\ga_2\|_{C[0,T]} & =\|\bm\rho\|_{C[0,T]}\le C\|\pa_t^\al\bm h\|_{C[0,T]}\le C\left(\sum_{j=1}^d\|\pa_t^\al w(\bm x^j,\,\cdot\,)\|_{C[0,T]}^2\right)^{1/2}\\
& \le C\sum_{j=1}^d\|\pa_t^\al(u_1-u_2)(\bm x^j,\,\cdot\,)\|_{C[0,T]}.
\end{align*}
This completes the proof of Theorem \ref{thm-local}.
\end{proof}

\begin{remark} {In three dimensions, if the source moves on the plane $\{x_3=C\}$ where $C\in \BR$ is known, then two observation points are sufficient to imply the stability. Analogously,
if $d'\ (1\le d'<d, d>1)$ components of the orbit function are known, then the number of observation points can be reduced to $d-d'$.}
\end{remark}

\begin{proof}[Proof of Corollary $\ref{cor-global}$]
Fix the set $X$ of observation points and the constant $\ve$ in the assumption of Corollary \ref{cor-global}. For any $\bm\ga\in\cU_0$, by $\|\bm\ga'\|_{C[0,T]}\le K$ we know that $\bm\ga(t)\in\ov{B_\ve}$ for $t\le\ve/K$. Then we define
\[
T_\ell=\begin{cases}
\ve\ell/K, & \ell=0,1,\ldots,\lceil K T/\ve\rceil-1,\\
T, & \ell=\lceil K T/\ve\rceil
\end{cases}
\]
and consider the intervals $[T_{\ell-1},T_\ell]$ ($\ell=1,\ldots,\lceil K T/\ve\rceil$) successively.

We adopt an inductive argument and start from $\ell=1$. On $[T_0,T_1]=[0,T_1]$, the above observation implies $\|\bm\ga_i\|_{C[0,T_1]}\le\ve$ ($i=1,2$). Taking $\bm y=\bm0$ in \eqref{eq-assume-obs'}, we see that there exist $d$ observation points $\{\bm x^j(\bm0)\}_{j=1}^d\subset X$ satisfying \eqref{eq-assume-obs}. {Observing} that all assumptions in Theorem \ref{thm-local} are fulfilled, we utilize the uniqueness result in Theorem \ref{thm-local} to conclude that {the relation} $u_1(\bm x^j(\bm0),\,\cdot\,)=u_2(\bm x^j(\bm0),\,\cdot\,)$ ($j=1,\ldots,d$) on $[0,T_1]$ implies $\bm\ga_1=\bm\ga_2$ on $[0,T_1]$.

For general $\ell=2,\ldots,\lceil K T/\ve\rceil$, we make the induction hypothesis that the relation $u_1(\bm x^j,\,\cdot\,)=u_2(\bm x^j,\,\cdot\,)$ ($j=1,\ldots,N$) on $[0,T_{\ell-1}]$ implies $\bm\ga_1=\bm\ga_2$ on $[0,T_{\ell-1}]$. By the well-posedness of the forward problem, we have $u_1=u_2$ in $\Om\times[0,T_{\ell-1}]$. Introducing
\[
w_\ell(\bm x,t):=(u_1-u_2)(\bm x,t+T_{\ell+1}),
\]
we immediately see that $w_\ell$ satisfies the equation
\[
(\pa_t^\al+\cL)w_\ell=g(\bm x-\bm\ga_1(t+T_{\ell-1}))-g(\bm x-\bm\ga_2(t+T_{\ell-1}))
\]
in $(\bm\ga_1(T_{\ell-1})+\Om)\times(0,T-T_{\ell-1})$ with the homogeneous initial(-boundary) condition. Repeating the same argument as that in the proofs of Theorem \ref{thm-local} and the case of $\ell=1$, we can take $\bm y=\bm y_\ell:=\bm\ga_1(T_{\ell-1})$ in \eqref{eq-assume-obs'}, so that again we can find $\{\bm x^j(\bm y_\ell)\}_{j=1}^d\subset X$ such that \eqref{eq-assume-obs} is fulfilled with $\bm x^j$ replaced by $\bm x^j(\bm y_\ell)-\bm y_\ell$ ($j=1,\ldots,d$). Since all assumptions in Theorem \ref{thm-local} are satisfied in $(\bm\ga_1(T_{\ell-1})+\Om)\times(0,T_\ell-T_{\ell-1})$, we conclude that $w_\ell(\bm x^j(\bm y_\ell),\,\cdot\,)=0$ ($j=1,\ldots,d$) on $[0,T_\ell-T_{\ell-1}]$ implies $\bm\ga_1-\bm\ga_2=\bm0$ on $[T_{\ell-1},T_\ell]$ or equivalently,
\begin{equation}\label{eq-induction}
u_1(\bm x^j(\bm y_\ell),\,\cdot\,)=u_2(\bm x^j(\bm y_\ell),\,\cdot\,)\ (j=1,\ldots,d)\mbox{ on }[T_{\ell-1},T_\ell]\mbox{ implies }\bm\ga_1=\bm\ga_2\mbox{ on }[T_{\ell-1},T_\ell].
\end{equation}
By the {inductive argument,  for any } $\ell=1,\ldots,\lceil K T/\ve\rceil$ there exists a set of $d$ observation points $\{\bm x^j(\bm y_\ell)\}_{\ell=1}^d\subset X$ ($\bm y_\ell=\bm\ga_1(T_{\ell-1})$) such that \eqref{eq-induction} holds. Consequently, the proof is completed by collecting the uniqueness on all intervals $[T_{\ell-1},T_\ell]$.
\end{proof}

\section*{Acknowledgement}
This work is partly supported by the A3 Foresight Program ``Modeling and Computation of Applied Inverse Problems'', Japan Society for the Promotion of Science (JSPS) and National Natural Science Foundation of China (NSFC).
G. Hu is supported by the NSFC grant (No. 11671028) and NSAF grant (No. U1930402).
Y. Liu and M. Yamamoto are supported by JSPS KAKENHI Grant Number JP15H05740.
M. Yamamoto is partly supported by NSFC (Nos. 11771270, 91730303) and RUDN University Program 5-100.

\end{document}